\documentclass[12pt]{amsart}
\usepackage{bbm,bm}
\usepackage{amssymb}
\usepackage[margin=1.5in, total={6.0in, 8in}]{geometry}
\usepackage{colordvi}
\usepackage{graphicx,xspace}
\usepackage{color}
\usepackage{xspace}

\usepackage[bookmarks=false,colorlinks,linkcolor=blue,citecolor=green]{hyperref}
\newcounter{FNC}[page]
\def\fauxfootnote#1{{\addtocounter{FNC}{2}$^\fnsymbol{FNC}$%
     \let\thefootnote\relax\footnotetext{$^\fnsymbol{FNC}$\Magenta{#1}}}}
\numberwithin{equation}{section}
\newtheorem{theorem}{Theorem}[section]

\newtheorem{lemma}[theorem]{Lemma}
\newtheorem{corollary}[theorem]{Corollary}

\newtheorem{thm}[theorem]{Theorem}
\newtheorem{defn}[theorem]{Definition}
\newtheorem{exm}[theorem]{Example}
\newtheorem{rem}[theorem]{Remark}
\author{Cassandra Durell}

\author{Stefan Forcey} \address[S. Forcey]{
    Department of Mathematics\\
    The University of Akron\\
    Akron, OH 44325-4002
    }
    \email{sf34@uakron.edu}  \urladdr{http://www.math.uakron.edu/\~{}sf34/}

\title[Level-1 Network Polytopes]{Level-1 Phylogenetic Networks and their Balanced Minimum Evolution Polytopes}

\begin{document}

\begin{abstract}
    Balanced minimum evolution is a distance-based criterion for the reconstruction of phylogenetic trees. Several algorithms exist to find the optimal tree with respect to this criterion. One approach is to minimize a certain linear functional over an appropriate polytope. Here we present polytopes that allow a similar linear programming approach to finding phylogenetic networks.   We investigate a two-parameter family of polytopes that arise from phylogenetic networks, and which specialize to the Balanced Minimum Evolution polytopes as well as the Symmetric Travelling Salesman polytopes. We show that the vertices correspond to certain level-1 phylogenetic networks, and that there are facets or faces for every split. We also describe minimal facets and a family of faces for every dimension.
 \end{abstract}
\maketitle

%%%%%%%%%%%%%%%%%%%%%%%%%%%%%%%%%%%%%%%%%%%%%%%%%%%%%%%%%%%%%%%%%%%%%%%%%%%%%%%%%%%%%
\section{Introduction}
In balanced minimum evolution (BME) methods we try to minimize the total branch length of a candidate phylogenetic tree for a given discrete metric. The BME method can be described as a linear programming problem. The convex hull of solutions for this problem, given $n$ taxa, is the ${n \choose 2}-n$ dimensional polytope BME(n). In this paper we present a generalization of the BME polytopes which allows the solutions to be  phylogenetic networks rather than  trees.

The main results in this paper are theorems about the convex polytopes we define, denoted BME($n,k$) for all $0\le k \le n-3$. We prove that their vertices correspond bijectively to binary level-1 phylogenetic networks with $n$ leaves and $k$ non-trivial bridges, in Theorem~\ref{weight}, Theorem~\ref{min}, and Corollary~\ref{vert}. In Theorem~\ref{t:vertcount} we find a formula to count the vertices:
$$
{n-3 \choose k}\frac{(n+k-1)!}{(2k+2)!!}.
$$ Section~\ref{basics} contains the basic definitions we need.  In Section~\ref{poly} we define and compare the new polytopes to well-known families: the symmetric travelling salesman polytopes and the balanced minimum evolution polytopes. Ours are nested inside the STSP and outside the BME polytopes (after  scaling) as shown in Theorem~\ref{t:nest}. In Section~\ref{faces} we describe a good deal of the facial structure of our new polytopes. Every less-refined phylogenetic network turns out to correspond to multiple faces: one in each polytope with vertices which refine that network. See Theorem~\ref{t:faces} for details. In Section~\ref{s:facet} we show that some of these faces are actually facets, specifically those corresponding to splits as shown in Theorem~\ref{split_facet}. We also describe lower bound facets in Theorem~\ref{lb_facet}. Finally in Section~\ref{5:1} we describe further results about the specific case of networks with five leaves and a single bridge.  We give a complete classification of the 62 facets in this case, in Theorem~\ref{2splits}, Theorem~\ref{lb_51}, Theorem~\ref{exclud},  and Theorem~\ref{cyc51}.

\section{Basics}\label{basics}
We begin with the set $[n] = \{1,2,...,n\}$, in which the integers 1 through $n$ stand for biological taxa; along with a given non-negative pairwise distance function, or dissimilarity matrix $\mathbf{d}$, with entries denoted ${d}_{ij}= \mathbf{d}(i,j)$ for each pair of taxa $i,j\in[n]$. (Note that $\mathbf{d}$ can also be described as a discrete metric, or a distance vector.) We try to find an appropriate combinatorial structure to display that data. The structures we directly consider here are split networks and phylogenetic networks. We will restrict to the unrooted versions, and focus on specializations of one or both structures such as: phylogenetic trees, circular split networks, level-1 and 1-nested phylogenetic networks.

The simplest structure we consider is an (unrooted) \emph{phylogenetic tree}.  Mathematically this is a graph with no cycles, and with no nodes of degree 2. The nodes with degree larger than 2 are unlabeled, but the $n$ leaves are labeled bijectively with our $n$ taxa. There is also an option of assigning non-negative lengths to the edges of the tree. The appropriateness of this data structure is clear when the edge-lengths along the path between two leaves $i$ and $j$ sum to the given distance ${d}_{ij}$ between those leaves. In fact, if our given pair-wise distances allow such a representation then the weighted tree is unique.

A \emph{split system} is a more general data structure,  (which specializes to a phylogenetic tree). A \emph{split} of our set is a partition $A|B$ of $[n]$ into two parts. When one part has cardinality 1 we call the split \emph{trivial}.   A split system $s$ on $[n]$ is any collection of splits which contains all the trivial splits.  We say a split system $s'$ \emph{refines}  $s$ when $s' \supset s.$
A \emph{split network} is a graphical representation of a split system. It is a special connected simple graph. The $n$ taxa are again seen as the labeled leaves (degree 1 vertices). Each split is represented by a set of parallel edges which is a minimal cut of the graph; that is, removing that set of edges separates the graph into two components, whose respective leaves are the parts of the split. Sometimes there are (nontrivial) splits represented by a single edge. We call these edges (nontrivial) \emph{bridges}.  A phylogenetic tree is a special split network. Its splits are all bridges. Upon removing a bridge the leaves of the resulting disconnected components are the parts of the split. The two parts of each split are called \emph{clades}. We say that the tree \emph{displays} those splits, or is \emph{consistent} with those splits.

\begin{figure}[h]
\centering{\includegraphics[width=\textwidth]{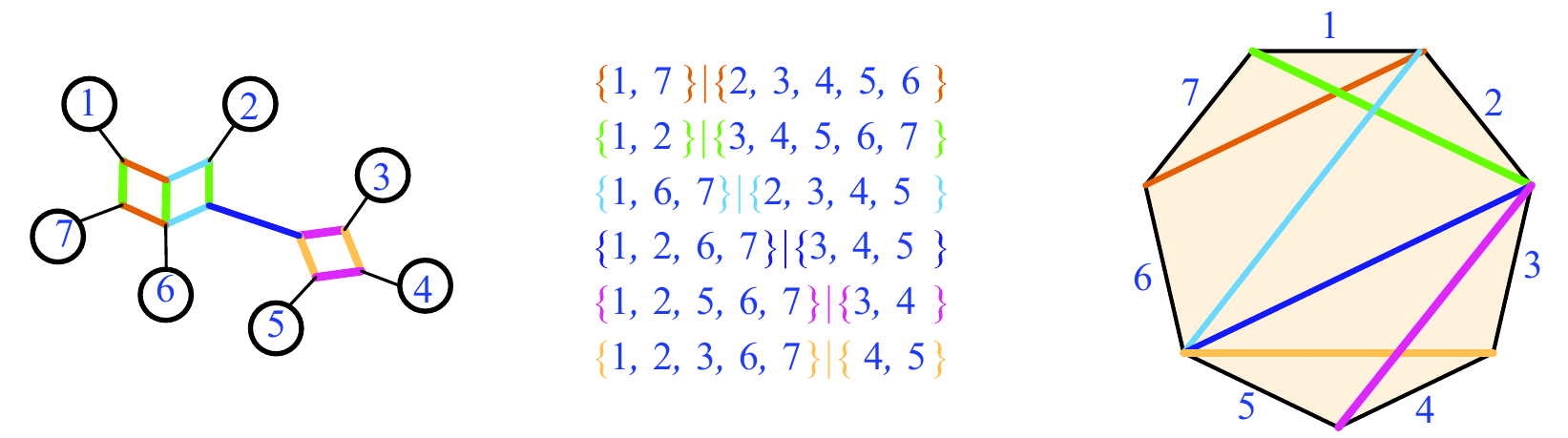}}
\caption{ A circular split system $s$, in the center (trivial splits not shown), with its circular split network on the left, and its polygonal representation on the right. This split network has one non-trivial bridge, giving the split $\{1,2,6,7\}|\{3,4,5\}.$ The split network is externally refined, so no bridges can be added.}
\label{f:network}
\end{figure}

  A \textit{circular}  split network is one whose graph may be drawn on the plane without edge crossings, and with its leaves all on the exterior of the diagram. The terminology is due to the fact that any such drawing automatically produces a circular ordering of the leaves. Note that these drawings are not fixed in the plane---twisting around a bridge gives a different diagram but represents the same split network.
As shown in \cite{dev-petti}, there is an equivalent \emph{dual polygonal  representation} of any circular
split network:  Given a circular split system with a circular
ordering $c$ of the species, consider a regular $n$-gon, with the
edges cyclically labeled according to $c$.  For each split, draw a
diagonal partitioning the appropriate edges; see Figure
\ref{f:network}. Note that splits which require multiple parallel edges in the the network picture correspond to diagonals
that are \emph{crossing} (they intersect other diagonals in the picture) and that bridges become \emph{noncrossing} diagonals.

\begin{defn}
An \emph{externally refined} split network $s$ is such that there is no split network $s'$ both refining $s$ and possessing more bridges than $s$.\end{defn} In the polygonal representation of an externally refined split network there are no non-crossing diagonals that can be added; that is, any additional split added to the network $s$ will correspond to a diagonal that intersects existing diagonals.
An externally refined phylogenetic tree has non-leaf nodes  that are all  of degree three; this is usually referred to as a \textit{binary}, or \emph{bifurcated} tree.
 Note that an externally refined split network can be a refinement of another externally refined split network.

The following definitions are from \cite{Gambette2017}. Another generalization of an unrooted phylogenetic tree is an (unrooted) \emph{phylogenetic network}: this is a simple connected graph with exactly $n$ labeled nodes of degree one, all other unlabeled nodes of degree at least three, and every cycle of length at least four. (Cycles here are simple cycles, with no repeated nodes other than the start.) If every edge is part of at most one cycle then the network is called \emph{1-nested}. If every node is part of at most one cycle, the network is called \emph{level-1}.  If that is true and  all the unlabeled non-leaf nodes also have degree three, then the network is called \emph{binary level-1}. Level-1 networks, as a set, include level-0 networks, which are the phylogenetic trees. The level-1 and 1-nested networks are special versions of galled trees, (which sometimes allow cycles of length three as in \cite{semple_steel_uni}), and of cactus models (which allow labels for non-leaves as in \cite{brandes}).

%Level-1 networks are also closely related to $PC$-trees as studied in https://cloudfront.escholarship.org/dist/prd/content/qt6cp9n3nj/qt6cp9n3nj.pdf?t=mtfydz and http://www.cs.colostate.edu/~rmm/pctrees.pdf.  A $PC$-tree for which the all nodes are cyclic (class $C$) except for the nodes of degree 3 (which are permutable, class $P$.)

Notice that a phylogenetic tree is both a split network and a level-0 phylogenetic network. In contrast to the split networks, the phylogenetic networks do not have parallel sets of edges, but sets of edges are still used to represent splits.
%Given a set of taxa, we say that a phylogenetic network on those %taxa \emph{contains} those phylogenetic trees (with the same taxa) %that are obtained by contracting or removing some edges of the %network. That same terminology is used when the containing network %is a split network, but then contractions or removals must be %applied to splits---so sets of parallel edges all at once.
 A \emph{minimal cut} $C$ of a phylogenetic network is a subset of the edges which, when removed, leaves two connected components. The edge set is minimal in the sense that no more edges are removed than is necessary for the disconnection. The split $A|B$ displayed by such a cut is the two sets of leaves of the two connected components.  A split $A|B$ is \emph{consistent} with a phylogenetic network if there is a minimal cut $C(A|B)$ displaying that split.  A split system $s$ is consistent with a phylogenetic network if all its splits are, and the (maximal) system of all such splits for a phylogenetic network $N$ is called $\Sigma(N).$
 In Figure~\ref{my_splits_ex_SN} we show a binary level-1 network $N$, and the associated maximal split network $\Sigma(N)$.
 Multiple different phylogenetic networks can map to the same split system  under the mapping $\Sigma$.

 \begin{figure}[h]
\centering{\includegraphics[width=\textwidth]{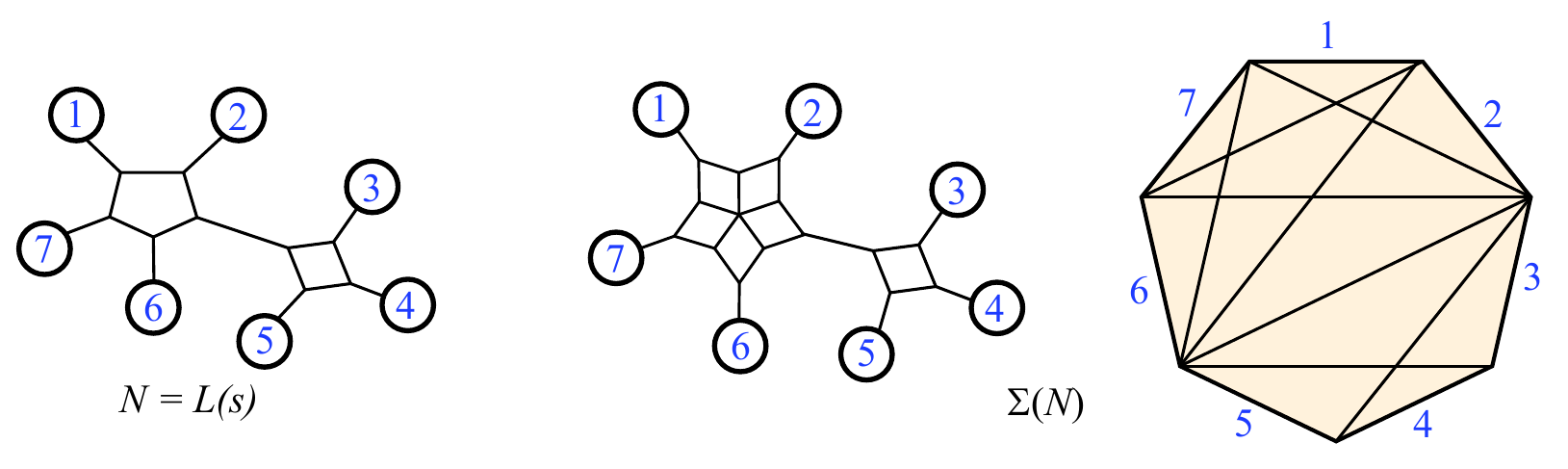}}
\caption{ A level-1 phylogenetic network $N$ with its associated maximal circular split system $\Sigma(N)$, shown both as a network and polygonal representation. Here $N$ is the image $L(s)$ for $s$ the split network in Figure~\ref{f:network}.}
\label{my_splits_ex_SN}
\end{figure}

 There is an even closer relationship between level-1 (and thus 1-nested) networks and circular split networks. In \cite{Gambette2017} it is shown that a split network $s$ is circular if and
only if there exists an unrooted level-1 network $N$ such that $s \subset \Sigma(N).$ For instance the split network $s$ in Figure~\ref{f:network} has splits a subset of those in $\Sigma(N)$ seen in Figure~\ref{my_splits_ex_SN}.

If $s$ is a circular split system then there is a simple way to associate to $s$ a specific 1-nested phylogenetic network denoted as $L(s)$.
\begin{defn}
  Construct this network $L(s)$ as follows:  begin with a split network diagram  of $s$ and  consider the diagram as a planar drawing of its underlying planar graph, with  leaves on the exterior.  Then 1) delete all the edges that are not adjacent to the exterior of that graph, and 2) smooth away any resulting degree-2 nodes.
\end{defn}

See $s$ in Figure~\ref{f:network} and $L(s)$ in Figure~\ref{my_splits_ex_SN}.
We see that $L(s)$ displays all the splits of $s$, and has the same bridges as $s$. Explicitly, any bridge of $s$ is displayed by a bridge in the image, and any other split in $s$ by a pair of edges in a cycle of $L(s)$.
Refinement is seen easily in the pictures of split networks, by collapsing parallel sets of edges or by removing diagonals from the polygon. The function $L$ preserves refinement in phylogenetic networks.  See Figure~\ref{my_splits_ex_cont} for an example.  If $s$ is an externally refined split network then $L(s)$ is a binary level-1 phylogenetic network.
\begin{figure}[h]
\centering{\includegraphics[width=\textwidth]{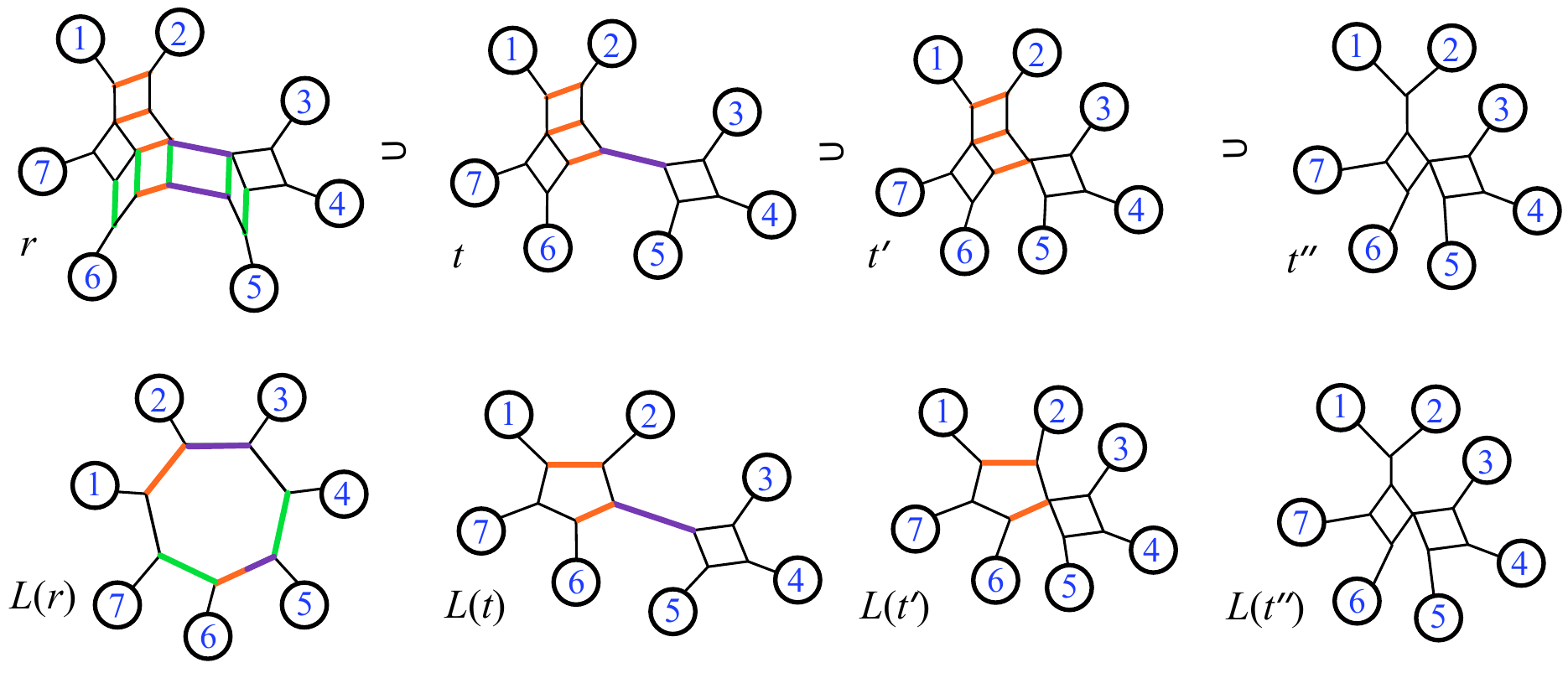}}
\caption{Here $r$  is a split network that refines $t$, which in turn refines $t'$, which refines $t''$. Also $\Sigma(L(t'')) \subset  \Sigma(L(t')) \subset \Sigma(L(t)) \subset \Sigma(L(r)).$}
\label{my_splits_ex_cont}
\end{figure}
Note that in \cite{Gambette2017} the authors define a similar function called $N$. Their function is equivalent to ours, if the definition in \cite{Gambette2017} is modified so as not to depend on $k$-marguerites.

%%%%%%%%%%%%%%%%%%%%%%%%%%%%%%%%%%%%%%%%%%%%%%%%%%%%%%%%%%%%%%%%%%%%%%%%%%%%%%%%%%%%%
\section{Polytopes}\label{poly}

A circular ordering $c$ is \textit{consistent} with a circular split system $s$ if a planar network of $s$ may be drawn, such that the leaves lie on the exterior in the order given by $c$. Similarly a circular ordering $c$ is \textit{consistent} with a level-1 phylogenetic network $N$ if a planar diagram of $N$ may be drawn, such that the leaves lie on the exterior in the order given by $c$. Thus a circular ordering $c$ is consistent with $s$ if and only if $c$ is consistent with $L(s).$

We define new families of polytopes by assigning  vectors to each externally refined circular split network $s$, and thus to each binary level-1 phylogenetic network.
\begin{defn} The vector ${\mathbf x}(s)$ is defined to have lexicographically ordered components ${x}_{ij}(s)$ for each unordered pair of distinct leaves $i,j \in [n]$ as follows:

\begin{equation} \label{e:bmenkvert} {x}_{ij}(s) = \begin{cases} 2^{k-b_{ij}} & \text{if there exists $c$ consistent with $s$; with $i,j$  adjacent in $c$,}\\ 0 & \text{otherwise.} \end{cases} \end{equation}

where $k$ is the number of bridges in $s$ and $b_{ij}$ is the number of bridges crossed on any path from $i$ to $j$.\end{defn}

The formula for ${\mathbf x}(N)$ works just as well when $N$ is a binary level-1 network. Indeed, we clearly have for any externally refined circular split network that $${\mathbf x}(L(s)) = {\mathbf x}(s).$$
  The vector is determined entirely by the number and placement of the bridges. Thus two split systems with the same associated binary level-1 network will have the same vector ${\mathbf x}.$ Therefore we will often use ${\mathbf x}(L(s))$ interchangeably with ${\mathbf x}(s),$ for $s$ any preimage of $L(s).$

\begin{figure}[h]
\centering{\includegraphics[width=\textwidth]{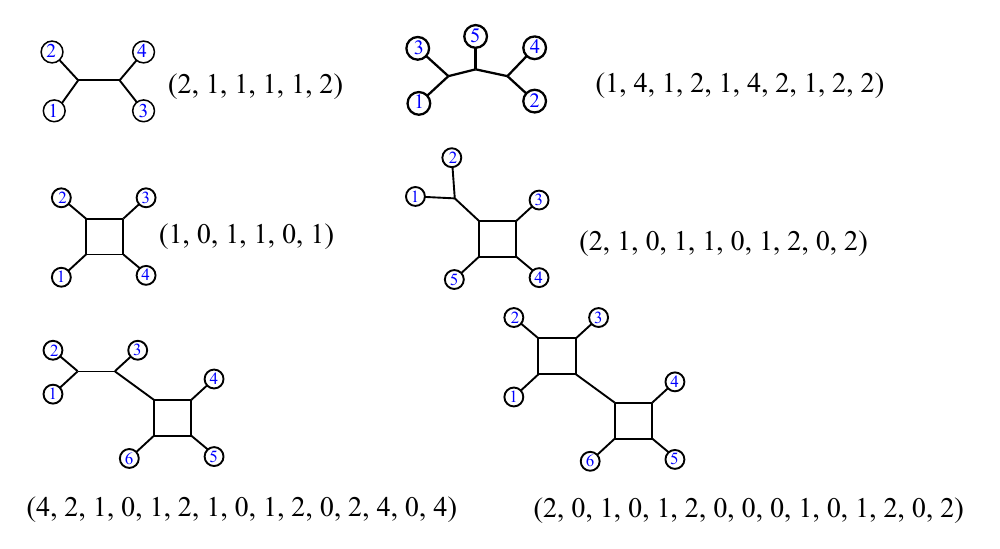}}
\caption{ Calculating some vectors $\mathbf{x}(s)$.}
\label{ex_vec}
\end{figure}

\begin{defn}
The convex hull of all the vectors ${\mathbf x}(s)$ for $s$ any externally refined circular network with $n$ leaves and $k$ nontrivial bridges is the \emph{level-1 network polytope} BME($n,k$). (The set of vectors is the same as the collection ${\mathbf x}(N)$ for $N$ any binary level-1 network with $n$ leaves and $k$ nontrivial bridges.)
\end{defn}

Before proving theorems about these newly discovered level-1 network polytopes, we relate them to well-known examples.  We begin with a review of the Balanced Minimal Evolution Polytopes BME($n$), and the Symmetric Travelling Salesman polytopes STSP($n$). The BME polytopes were first studied in 2008 \cite{Rudy2008}. We have found a simple description of the vertices as follows:

\begin{defn}
For each given binary phylogenetic tree $t$ with $n$ leaves and $k=n-3$ (nontrivial) bridges, the \textit{vertex vector} $\mathbf{x}(t)$ has ${n \choose 2}$ components
\begin{equation} \label{e:bmevert}
    x_{ij}(t) = 2^{k-b_{ij}}
\end{equation} where $b_{ij}$ is the number of nontrivial bridges on the tree from leaf $i$ to a different leaf $j.$

The convex hull of all the  $(2n-5)!!$ vertex vectors (for
all binary trees $t$ with $n$ leaves), is the polytope BME$(n),$  of dimension ${n \choose 2}-n.$
\end{defn}

Note that in \cite{forcey2015facets} the formula is given, equivalently, as $x_{ij}(t) = 2^{n-2-l_{ij}}$ where $l_{ij}$ is the number of internal nodes along the path from $i$ to $j$.

\subsection{Facets of BME($n$)}
The (lower-dimensional) \textit{clade faces} of BME ($n$) were described in \cite{Rudy}. Recently we have discovered  large
collections of (maximum dimensional) facets for all $n$, in \cite{forcey2015facets} and \cite{splito}.
  In the following list  we review our new facets, and show their statistics in Table~\ref{facts}.

\begin{enumerate}
\item Any
 split of $[n]$ with both parts larger than 3 corresponds to a facet of BME$(n)$, with vertices all the trees displaying that split.
\item A \emph{cherry} is a clade with two leaves. For each intersecting pair of cherries $\{a,b\},\{b,c\}$, there is a facet of BME$(n)$ whose
vertices correspond to trees having either cherry.
\item For each pair of leaves $\{i,j\}$, the caterpillar trees with that pair fixed at opposite ends constitute the vertices of a facet. These bound BME$(n)$ from below.
\end{enumerate}
%%%%%%%%%%%%%%%%%%%%%%%%%%%%%%%%%%
%
% table
%
%%%%%%%%%%%%%%%%%%%%%%%%%%%%%%%%%%
\begin{table}[hb!]
\begin{tabular}{|c|c|c|c|}
\hline
facets & facet inequalities & number of & number of \\
of BME$(n)$ && facets &   vertices \\
&&& in facet\\
  \hline \hline
 \rule{0pt}{2.6ex}\rule[-1.2ex]{0pt}{0pt}   Caterpillar & $x_{ab}\ge 1$ &${n \choose
2}$& $(n-2)!$\\
  \hline
 \rule{0pt}{2.6ex}\rule[-1.2ex]{0pt}{0pt}  intersecting-& $~x_{ab}+x_{bc}-x_{ac} \le 2^{n-3}$ &${n \choose 2}(n-2)$ & $2(2n-7)!!$ \\
 cherry&&&\\
 \hline \rule{0pt}{2.6ex}\rule[-1.2ex]{0pt}{0pt} $(m,3)$-split, $m\ge3$ & $x_{ab}+x_{bc}+x_{ac} \le 2^{n-2}$ & ${n \choose 3}$ & $3(2n-9)!!$\\
  \hline \rule{0pt}{2.6ex}\rule[-1.2ex]{0pt}{0pt} non-trivial $(m,p)$-split $A|B$ &
$\displaystyle{\sum_{i,j\in A} x_{ij} \le (m-1)2^{n-3}}$ &
$2^{n-1}-{n \choose 2}-n-1$ &
 $(2m-3)!!(2p-3)!!$\\
  \hline
\end{tabular}\caption{Known facets for the BME polytopes, BME($n$) = BME($n,n-3$). The third is a special case of the fourth.
 The inequalities are given for any $a,b,c,\dots \in [n].$  \label{facts}}
\end{table}

\subsection{STSP} Next we recall the travelling salesman polytopes. We consider symmetric tours (circular orderings of the $n$ taxa). These can be pictured as placing the numbers in order on a circle in which the orientation is not specified---reading around the circle clockwise or counterclockwise gives the same circular ordering.
\begin{defn}
For each circular ordering $c$ on the set $[n]$, the \emph{incidence vector} ${\mathbf x}(c)$  has ${n \choose 2}$ components. The components are
\begin{equation} \label{e:stspvert}
  x_{ij}(c)= \begin{cases} 1 & \text{if  $i$ and $j$ are adjacent in $c$}\\ 0 & \text{if not.} \end{cases} \end{equation}
  The  \emph{Symmetric Traveling Salesman Polytope}, denoted as STSP($n$), is the convex hull of these $\frac{(n-1)!}{2}$ vertex vectors. It has dimension ${n \choose 2}-n.$
\end{defn}

The two best-known sets of facets of the STSP are the subtour-elimination facets and the lower bound facets. The latter are given by the inequalities $x_{ij}\ge 0.$ Subtour-elimination facets correspond to any nontrivial split $A|B$ of $[n].$  The circular orderings which make up the vertices of such a facet are those which contain the elements of $A$ as a contiguous list, and thus $B$ likewise. There can be only two connecting edges between the parts of the splits. Requiring that tours be Hamiltonian means that there must be at least two such connecting edges in any tour, which eliminates the possibility of a subtour through one or the other. Thus for a given split, a facet defining inequality is
$$\sum_{i\in A, j\in B} x_{ij} \ge 2.$$
We show an alternate inequality in Theorem~\ref{split_facet}. (Recall that the STSP is of smaller dimension than its ambient space, allowing choices of inequality for all its faces.)

Clearly the three definitions~\ref{e:bmevert},~\ref{e:stspvert}, and~\ref{e:bmenkvert} of  ${\mathbf x}(s)$ agree when their input structures overlap. Circular orderings (seen as unicyclic level-1 networks) have no bridges, so when leaves are adjacent the exponent becomes 0. Trees (seen as networks) allow any two leaves to be adjacent, so the components of ${\mathbf x}$ are all nonzero in that case. Thus we see that restricting BME($n,k$) to the phylogenetic trees,  where $k=n-3$, recovers the polytopes BME($n$). Restricting BME($n,k$) to the fully reticulated networks, where $k=0$,  recovers STSP($n$).
Next we characterize the vector ${\mathbf x}$ from a combinatorial viewpoint.

\begin{figure}[h]
\centering{\includegraphics[width=\textwidth]{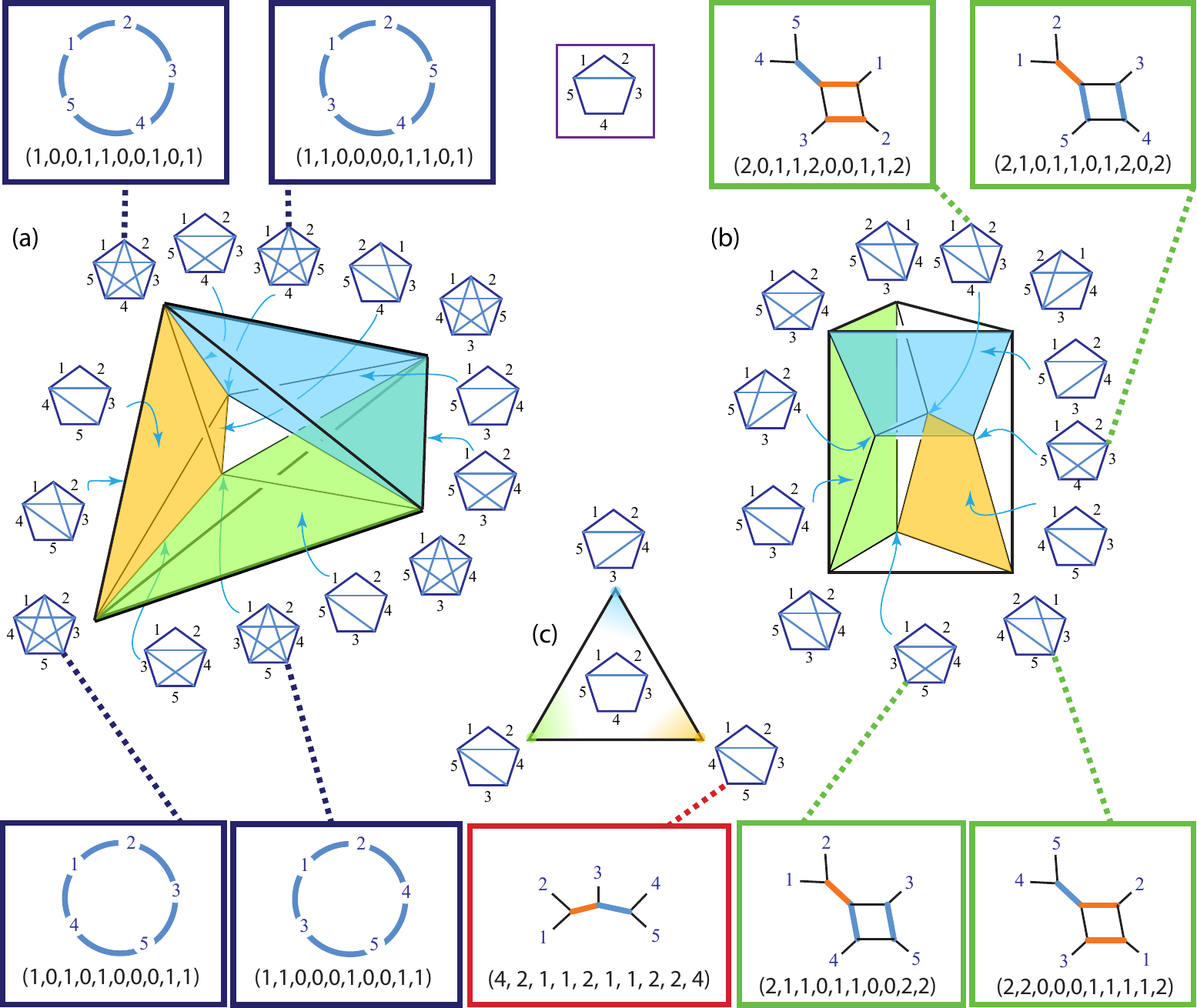}}
\caption {A facet (a) in BME(5,0) = STSP(5) , a facet (b) in BME(5,1), and a face (c) in BME(5,2) = BME(5). Summing either of the horizontal or  vertical pairs of vectors shown in (a) gives the corresponding vector shown in (b).  Summing all four vectors in (a) gives the vector shown in (c)}
\label{f:filter}
\end{figure}

\begin{thm}\label{cyc_vec}
For any externally refined circular split network $s$ with $k$ bridges, we have $${\mathbf x}(s) = \sum_{c  \text{ consistent} \atop \text {with } s} {\mathbf x}(c)$$ where the sum is over the exactly $2^k$ circular orderings $c$ consistent with $s$.
 Equivalently the component ${x}_{ij}(s)$
is the number of circular orderings consistent with that network for
which $i$ and $j$ are adjacent.
\end{thm}

\begin{proof}
 We show this equality by considering the sums of respective components. The components of $\mathbf{x}(c)$ are always 1 or 0. Note that the only way to alter a circular ordering $c$ which is consistent with the externally refined $s$, to another such consistent $c'$, is to choose a nontrivial bridge and twist the graph around that bridge. That is, we redraw the graph with one side of the bridge reflected vertically; we call this twisting the bridge. That there are exactly $2^k$ circular orderings contributing over all is seen by independently twisting all $k$ nontrivial bridges--each bridge contributes two options.  If $c$ is a circular ordering consistent with $s$, and further $c$ has $i$ and $j$ adjacent, then the bridges between $i$ and $j$ cannot be twisted without losing this adjacency. However, the other bridges ($k-b_{ij}$ of them) may be independently twisted while preserving the adjacency of $i$ and $j$. Thus each of those latter bridges contributes a factor of 2 to the total count of the consistent circular orderings with $i$ and $j$ adjacent: upon summing we thus achieve the defined value of $x_{ij} = 2^{k-b_{ij}}$. Examples are seen in Figure~\ref{f:filter}.
\end{proof}

\begin{corollary}\label{compsum}
We can infer that for any network $s$ the sum of all the components of $\mathbf{x}(s)$ obeys $\sum x_{ij} = n2^k.$ \end{corollary} \begin{proof}The total follows from the fact that the sum of components for any tour on $[n]$ is $n$, the number of edges in the tour.
\end{proof}
Closely related is the following \emph{twisting} lemma, useful in the next Section:
\begin{lemma}\label{twist}
Let $s$ have at least one bridge $b$. Then there exist two ways to add a single split to $s,$ to achieve by those additions two split networks $s'$ and $s''$ each with one less bridge than $s,$ and such that: $$\mathbf{x}(s) = \mathbf{x}(s') + \mathbf{x}(s'').$$
\end{lemma}
\begin{proof}
Notice that in the polygonal picture of a split network $s$ with $k+1$ bridges, for any given bridge $b$ there is always a way to add a split, (as a new diagonal), which crosses that bridge $b$ but no other existing bridge. This is true since (even if the maximum number of bridges is present), the bridge $b$ can be seen as one of the two diagonals of a quadrilateral which is nested inside the polygon. See Figure~\ref{f:twist}. The other diagonal of that quadrilateral is, of course, missing---so it is always available to become the new split which crosses $b$. That new diagonal crosses no other bridge since a quadrilateral can have only two diagonals.
\begin{figure}[h]
\centering{\includegraphics[width=\textwidth]{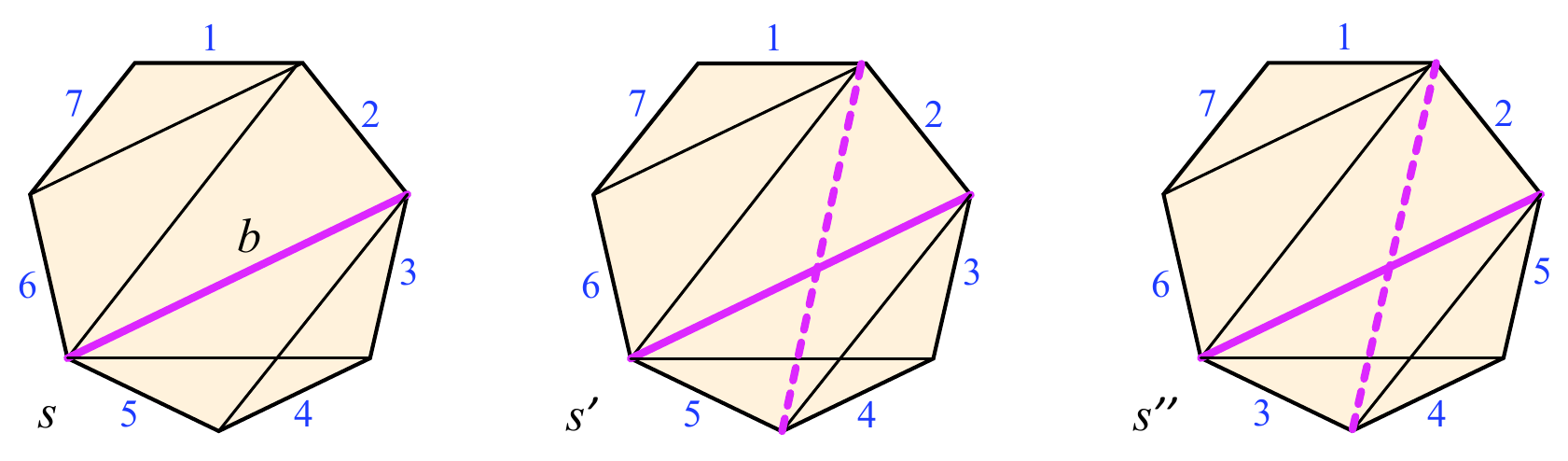}}
\caption {Example of the three split networks described in Lemma~\ref{twist}.}
\label{f:twist}
\end{figure}
 We use this new diagonal to reduce the number of bridges by one. We create $s'$ and $s''$ as  the two possible results of adding that new split which crosses $b$, with and without a twist around $b.$ From Theorem~\ref{cyc_vec} we have that ${\mathbf x}(s')$ is the sum of the vectors ${\mathbf x}(c)$ for all the circular orderings consistent with $s'$, and likewise for $s''$. The key here is that the circular orderings consistent with $s$ are partitioned into those consistent with $s'$ and those consistent with $s''.$

Thus the vector ${\mathbf x}(s)$ is the sum of the two vectors ${\mathbf x}(s')$ and ${\mathbf x}(s'').$

\end{proof}

\section{Metrics and faces}\label{faces} For our proofs about the structure of BME($n,k$), we need weighted networks.
A \emph{weighting} of a split system $s$ is a function $w:s\to\mathbb{R}_{\ge 0}$. In practice each split is assigned a positive weight, since when splits are assigned weight = 0 this system can be equated to the system minus those splits. Given such a weighted split system we can derive a metric $\mathbf{d}_s$ on $[n],$ where $$\mathbf{d}_s(i,j) = \sum_{\i\in A, j\in B} w(A|B)$$ where the sum is over all splits of $s$ with $i$ in one part and $j$ in the other. The metric is often referred to as the distance vector $\mathbf{d}_s.$  We can also derive a weighting on the edges of the 1-nested network, extending  our function $L$ to weighted networks $L(s)$. Here the weight function is from the edges of $L(s)$ to  positive real numbers, and given by $$w_{s}(e) = \sum_{e\in C(A|B)} w(A|B)$$ where the sum is over the splits of $s$ which are represented by a minimal cut containing $e.$ Clearly the sum of weights on a shortest path in $L(s)$ from $i$ to $j$ equals $\mathbf{d}_s(i,j).$ See Figure~\ref{f:weights} for an example.

\begin{figure}[h]
\centering{\includegraphics[width=\textwidth]{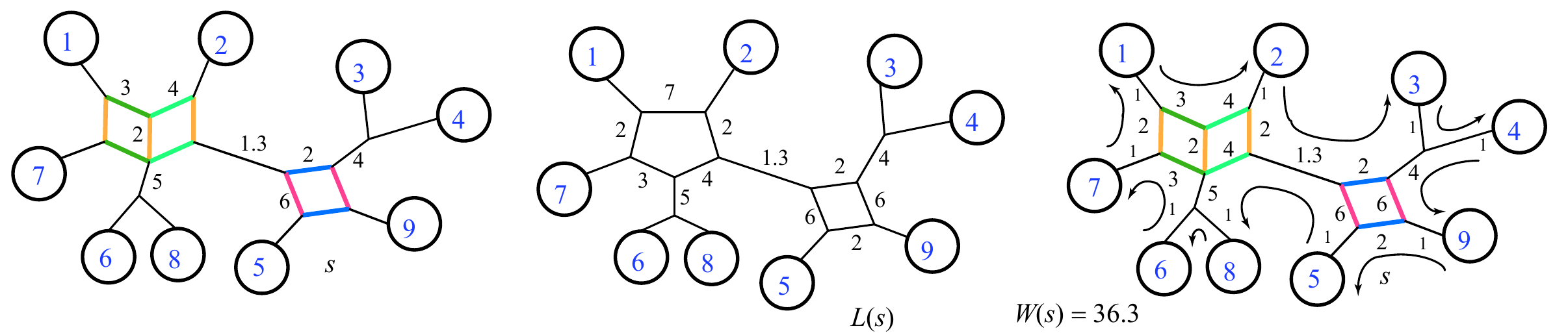}}
\caption {A weighted split network and its associated weighted level-1 network. Here the trivial splits (edges adjacent to leaves) are all given weight 1 for simplicity. The right-most picture is the same weighted network, with a tour shown by arrows. The tour length is twice the total weight, 72.6.  }
\label{f:weights}
\end{figure}

We define the total weight of the network to be the sum of all the weights: $$W(s) = \sum_{A|
B \in s} w(A|B).$$ Now we can still calculate $\mathbf{x}(s)$ for a weighted network, the vector does not depend on the weights. Instead we are interested in the dot product:
\begin{thm}\label{weight} For an externally refined weighted split network $s$, the dot product of our network vector with the distance vector gives a multiple of the sum of the weights: $$\mathbf{x}(s)\cdot \mathbf{d}_s = 2^{k+1}W(s)$$
\end{thm}

\begin{proof} Since $s$ is circular, given a network diagram for $s$, the distance $\mathbf{d}_s(i,j)$ for adjacent $i$ and $j$ can be found by adding the weights on edges between $i$ and $j$ on the exterior of the diagram. (That is also the same as the weights on the edges of $L(s)$ between them). For a circular ordering $c$ that is consistent with the split system $s$ we see that $\mathbf{x}(c)\cdot\mathbf{d}_s$ is equal to summing the distances between adjacent pairs of taxa, and that this sum includes the weight of each split exactly  twice: it totals to $2W(s)$. See Figure~\ref{f:weights} for an example. Since there are $2^k$ such circular orderings whose vectors sum to $\mathbf{x}(s),$ we have that  $\mathbf{x}(s)\cdot \mathbf{d}_s = 2(2^{k})W(s)$.
\end{proof}

\begin{thm}\label{min} Moreover, if $s'$ is any other binary level-1 network on $[n]$ with $k$ bridges then $\mathbf{x}(s')\cdot \mathbf{d}_s > \mathbf{x}(s)\cdot \mathbf{d}_s.$
\end{thm}
\begin{proof}
For a consistent circular ordering $c$ we have that $\mathbf{x}(c)\cdot\mathbf{d}_s = 2W(s),$ which is the length of a tour of $[n]$. Indeed it is a minimum length tour for the given metric $\mathbf{d}_s$, since visiting any of the taxa out of this order would involve retracing some part of the path between them. This is most clearly seen by considering the level-1 network $L(s)$ with weighted edges, where the distance between each pair is the minimum length path. See Figure~\ref{f:weights}. The circular ordering minimizing a tour length using those minimum length paths must be consistent with the network, or else some portion of some path between leaves will be traversed twice, increasing the length of the tour.  Thus since the dot product $\mathbf{c}(s)\cdot \mathbf{d}_s$ is minimized for each consistent $c,$ the sum of those products is minimized for the network $s.$ (The set of consistent circular orderings is determined uniquely by the network, and exchanging any $c$ for a non-consistent alternative would increase that term in the sum.)
\end{proof}

The previous two results do have a geometric interpretation, which is:
\begin{corollary}\label{vert}
\label{ccool} The vertices of BME($n,k$) are the vectors ${\mathbf x}(s)$ corresponding to the distinct binary level-1 networks $L(s)$.  That is for each externally refined circular network $s$, with $n$ leaves and $k$ bridges, we get a vertex of the polytope (but it is determined only by $L(s).$)
\end{corollary}

Remark:
Levy and Pachter \cite{Pachter2}, generalizing the
work of Semple and Steel \cite{Steel}, define a coefficient $\eta$ which takes values the components of our vector ${\mathbf x}(s)$. For an arbitrary distance vector $\mathbf{d}$, Levy and Pachter  call the dot product $\mathbf{x}(s)\cdot \mathbf{d}$ the \emph{length} of $\mathbf{d}$ with respect to $s.$  They point out that neighbor-net is a greedy algorithm for minimizing this quantity. Our results show how to minimize this length via linear programming. We also see as a consequence that that length is minimized precisely by a binary level-1 network, (or several if the number of bridges is larger than $k$. )

The question is raised: if the vertices of BME($n,k$) correspond to binary level-1 networks, but minimize a length that is a function of the weighted split network, then what role is left for an arbitrarily weighted 1-nested network? In fact, any weighted 1-nested network $N$ has the following property:
\begin{thm}
 If $\mathbf{d}_N$ is the metric on the leaves of $N$ defined by $\mathbf{d}_N(i,j)$ equal to the least sum of weights along a path between leaves $i$ and $j$, then there is a unique circular weighted split system $s= S_w(N)$ which has the same associated metric. That is, $\mathbf{d}_N = \mathbf{d}_s.$
\end{thm}
\begin{proof}
First we show that $\mathbf{d}_N$ obeys the Kalmanson condition: there exists a circular ordering of $[n]$ such that for all $1\le i<j<k<l\le n$ in that ordering, $$
\max\{\mathbf{d}_N(i,j)+\mathbf{d}_N(k,l), \mathbf{d}_N(j,k)+\mathbf{d}_N(i,l)\} \le \mathbf{d}_N(i,k)+\mathbf{d}_N(j,l).
$$
The circular ordering that meets our specifications is just any choice of one of the circular orderings consistent with $N.$ The two paths involved on the right hand side of the condition cross each other. Then since the leaves are on the exterior, the four paths involved on the left hand side of the condition are each bounded above in length by a path made by following first one crossing path and then the other, (switching at the crossroads, after their shared portion.) Two paths in a sum on the left hand side of the condition can at most use exactly all of both the crossing paths, so that the inequality is guaranteed. For example, in the following network $N$  we choose to look at the four taxa 1,2,8,7 in that order.  The crossing paths go from 1 to 8 and 2 to 7, with lengths of 12 and 11 respectively.  The graph edges used by the crossing paths are highlighted. Notice that the shortest path from 1 to 2, length 9, is bounded above by the path from 1 to 2 using highlighted  edges. The other three paths, from 2 to 8, from 8 to 7, and from 7 to 1, all actually  use highlighted edges borrowed from the crossing paths. In this case the inequality becomes
$\max\{19,17\}\le 23$.

\begin{center}
    {\includegraphics[width=2.5in]{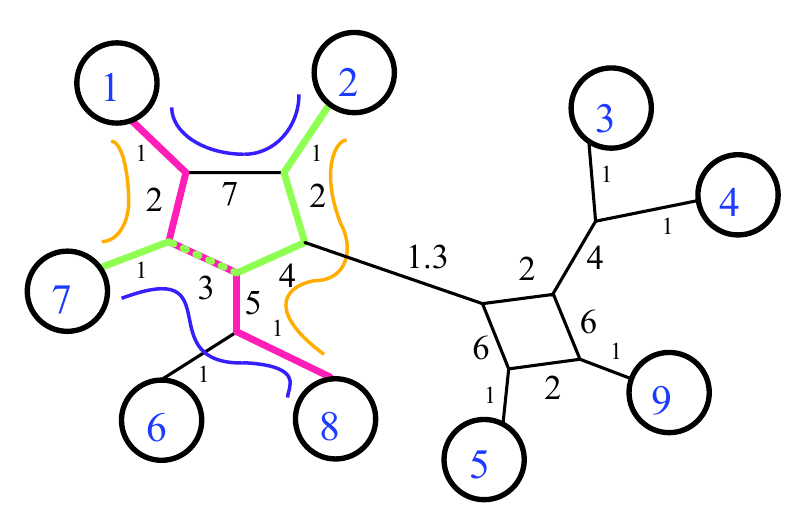}}
\end{center}

It is well known that for any Kalmanson metric $\mathbf{d}_N$ there exists a unique weighted split system $s$ whose weighting gives that metric: $\mathbf{d}_N = \mathbf{d}_s$.  See Chapter 6 in \cite{steelphyl}. To actually calculate this split system, the algorithm neighbor-net can be used; since it is guaranteed to return the unique answer for any Kalmanson metric, as shown in \cite{Bryant2007}.
\end{proof}

Now we show that for any 1-nested network, we get faces of our polytopes. In fact we get multiple faces from each network: one in each of the polytopes BME($n,k$) for which that network has more bridges than $k.$ Precisely:
\begin{thm}\label{t:faces}
Every $n$ leaved 1-nested network $S$ with $m$ bridges corresponds to a face $F_k(S)$
of each BME($n,k$) polytope for $0\le k\le m.$ That face has vertices all the  binary level-1 $k$-bridge networks $S'$ whose splits refine those of $S$, that is such that $\Sigma(S)\subset\Sigma(S')$.  \end{thm}

For example: in Figure~\ref{f:subfaces} there are three faces shown. The first two (a) and (b) are in BME(5,0), and the third (c) is in BME(5,1). They are pictured in context in Figure~\ref{f:filter}. Here we include the vector $\mathbf{d}_s$ for each of the three. In (a) we have $W(s) = 7$ and the vertices of BME(5,0) obey $\mathbf{x}(s') \cdot \mathbf{d}_s \ge 14.$  In  In (b) we have $W(s) = 8$ and the vertices of BME(5,0) obey $\mathbf{x}(s') \cdot \mathbf{d}_s \ge 16.$ In (c) we have $W(s) = 7$ and the vertices of BME(5,0) obey $\mathbf{x}(s') \cdot \mathbf{d}_s \ge 28.$ In the figure we show vertices that obey the inequality sharply. To see the strict inequality take dot products with any other vertex vector from the respective polytope.
 \begin{figure}[h]
\centering{\includegraphics[width=\textwidth]{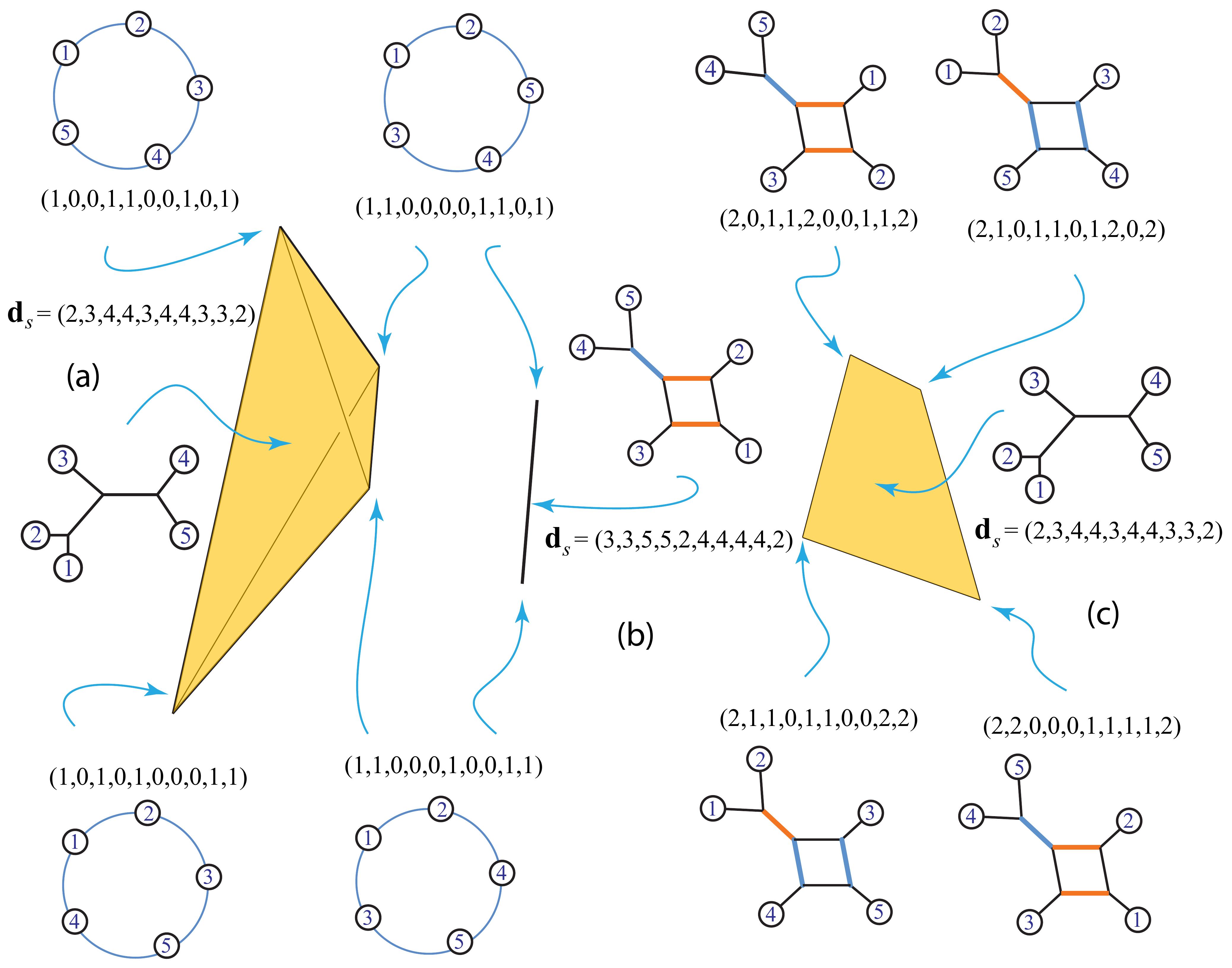}}
\caption {Three subfaces.}
\label{f:subfaces}
\end{figure}

\begin{proof} Without loss of generality we choose a split network $s$ which has the exterior form of $S,$ that is $L(s) = S.$ Let $s$ be weighted by assigning the value of 1 to each split. Then $W(s)$ is the total number of splits in $s$.
Let ${\mathbf d}_s$ be the distance vector derived from that weighting, so that the $i,j$ component of ${\mathbf d}_s$ is the number of splits between those leaves on $s.$ We see that the dot product ${\mathbf x}(s')\cdot{\mathbf d}_s$ is minimized simultaneously at each of the $k$-bridge networks $s'$ which externally refine $s$. In fact we have that the following inequality holds:
$$
{\mathbf x}(s')\cdot{\mathbf d}_s \ge 2^{k+1}W(s)
$$
for all $k$-bridge externally refined networks $s'$, and is an equality precisely when $s'$ refines $s.$
The reason is that ${\mathbf d}_s$ is equivalent to a distance vector derived from $s'$, where the splits are given weight $= 1$ if they are also in $s$, and weight $= 0$ if not. Thus the dot product will equal $2^{k+1}W(s)$ by the proof of Theorem~\ref{weight}, and will be minimized  by the proof of Theorem~\ref{min}.
\end{proof}

\begin{rem}
For a 1-nested network $s$ with $n$ leaves and $k$ bridges, embedded in each
polytope BME($n,j$) for $k > j \ge 0$ is a collection of faces corresponding to networks
  which refine our given network $s$.
Those faces link up (by sharing subfaces) to make an interesting complex, as shown by the shaded subfaces in Figure~\ref{f:filter}.  The topology of these complexes
is an interesting open question.\end{rem}

Two polytopes are called \emph{nested} when one lies inside the other, and the smaller polytope has all its vertices on the surface of the larger.  It turns out that all the level-1 network polytopes are, up to scaling, nested sequentially inside each other. Even more, they are all at the same time nested inside the Symmetric Travelling Salesman polytope. This is pictured in Figure~\ref{bary}. Precisely:
\begin{thm}\label{t:nest}
We can scale the network polytopes so that the polytope BME($n,k$) is nested inside BME($n,k-1$) for $0< k\le n-3.$ Furthermore, we can simultaneously scale all the BME($n,k$) polytopes so that they are all nested inside of STSP($n$) (with vertices at facial barycenters) and each BME($n,k$) is nested inside BME($n,j$) for $j < k.$
\end{thm}

\begin{proof}
Consider $0< j < k \le n-3.$ We show that the vertices of the scaled polytope $(2^{n-3-k})$BME$(n,k)$ lie on faces of the scaled polytope $(2^{n-3-j})$BME($n,j$) which in turn lie on the faces of the scaled polytope $(2^{n-3})$STSP($n$).

First, by Lemma~\ref{twist},
any vertex ${\mathbf x}(s)$ of BME$(n,k)$ is the sum of two vertices of BME$(n,k-1),$ found by adding a single split to $s$ in two ways.  Therefore if BME$(n,k-1)$ is first scaled by 2, the sum of those two vertices will also be scaled by 2. Thus the (unscaled) vertex ${\mathbf x}(s)$ of BME$(n,k)$ is the midpoint of those two scaled vertices: after adding, divide by 2. By convexity, this vertex is thus on a face of  (2)BME$(n,k-1).$

Secondly, any vertex ${\mathbf x}(s)$ of BME$(n,k)$  is on the surface of $(2^k)$STSP($n$).  To see this, recall from Theorem~\ref{cyc_vec} that ${\mathbf x}(s)$ is the sum of all ${\mathbf x}(c)$ for $c$ a circular ordering consistent with $s$.  Note that the consistent circular orderings are precisely the vertices of the face of STSP($n$) corresponding to the binary level-1 network $s.$ There are $2^k$ of them. Thus by first multiplying each by $2^k$ and then dividing their sum by $2^k$ we see that the vertex $\mathbf{x}(s)$ is at the barycenter of the corresponding face of $(2^k)$STSP($n$.)

Together these facts show that while each BME($n,k$) is nested in a scaled version of BME($n,k-1$), all are simultaneously nested in a scaled version of STSP($n$). By scaling STSP($n$) by $2^{n-3}$ and each BME($n,k$) by $2^{n-3-k}$ we can see them all nested simultaneously and sequentially. For example see Figure~\ref{bary}.
\end{proof}

\begin{figure}[h]
\centering{\includegraphics[width=\textwidth]{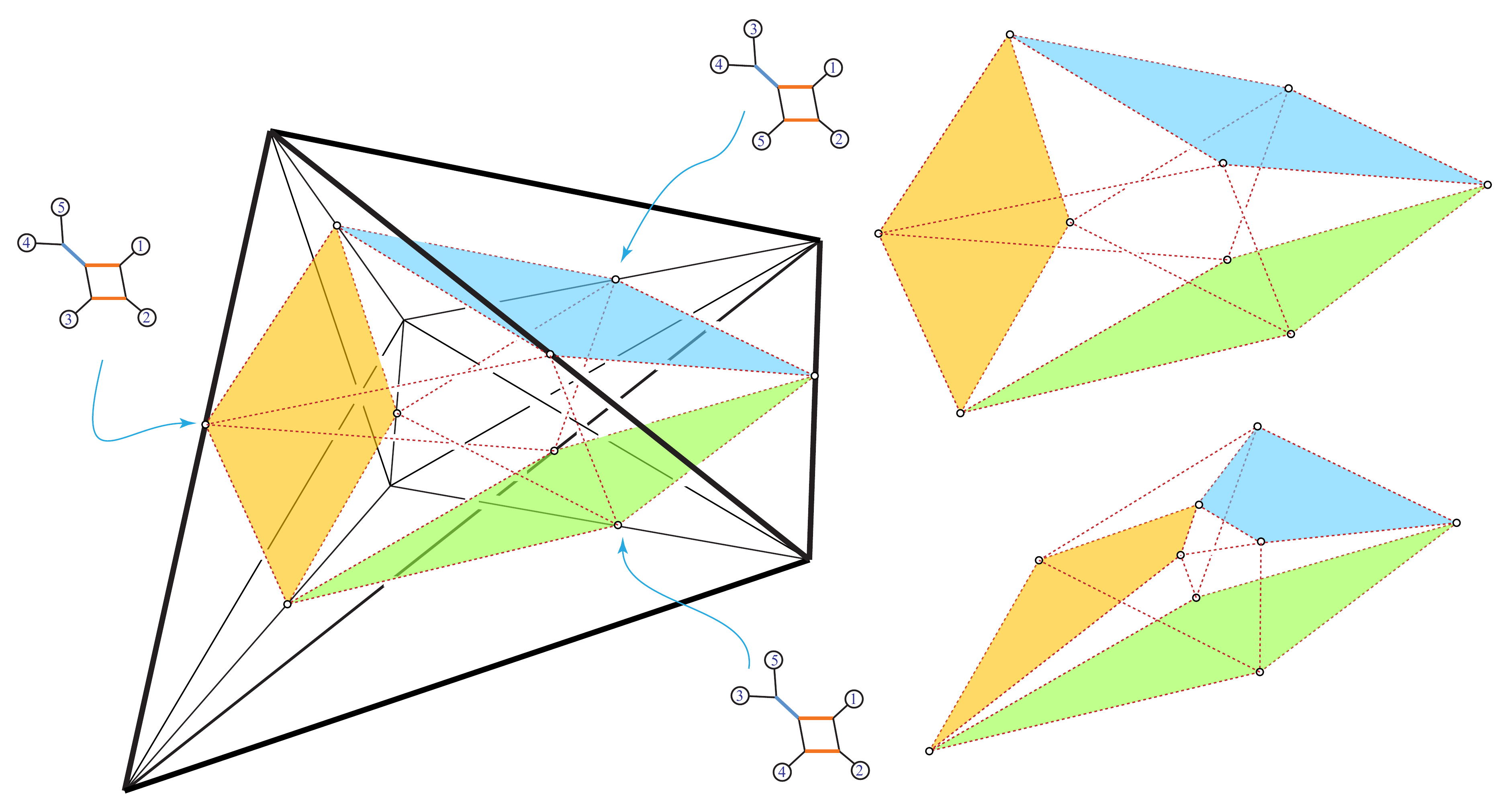}}
\caption{ The facet of STSP(5) pictured is the same one as in Figure~\ref{f:filter}(a), but scaled by a factor of 4. The corners of the shaded quadrilaterals are located at barycenters of faces (edges) of STSP(5). On the right we show just the convex hull of those barycenters, and then a Schlegel diagram obtained by rotating it. The vertices of this convex hull are those of the split facet of BME(5,1) shown in Figure~\ref{f:filter}(b) but scaled by 2.}
\label{bary}
\end{figure}

\begin{thm}\label{cool}
 The dimension of BME($n,k$) is ${n \choose 2}-n.$ The dimension reducing equalities are as follows: For each leaf $j=1,\dots,n$ the vertices ${\mathbf x}(s)$ satisfy $$\sum_{i\in [n] -\{j\}} x_{ij}= 2^{k+1}\, ,$$ where $k$ is the number of \emph{bridges} (non-crossing diagonals)
in the diagram.
\end{thm}

\begin{proof}
 The dimension results from the nesting property, Theorem~\ref{t:nest}. Since the STSP($n$) and BME($n$) polytopes are both of dimension ${n \choose 2}-n,$ the BME($n,k$) polytopes nested between them must also have that same dimension.

 The equalities generalize the \emph{Kraft} equalities for phylogenetic trees ($k = n-3$), and the \emph{degree} equations for the STSP($n$) for all $n$, ($k=0.$) In fact we can use the latter as base cases for the proof by induction on the number of bridges $k$. Assuming that the equality holds for all $s$ with $n$ leaves and $k$ bridges, we show that it holds for $s$ with $n$ leaves and $k+1$ bridges. Recall from Lemma~\ref{twist} that there are networks $s'$ and $s''$ with $k$ bridges and with $\mathbf{x}(s) = \mathbf{x}(s') + \mathbf{x}(s'').$ By our inductive assumption, both smaller vectors obey the required formula, both with the sum of components $=2^{k+1}.$ Thus after adding them together the resulting sum is $= 2^{k+2}.$
\end{proof}

\subsection{Counting binary level-1 networks}
The \emph{associahedra} $K(n)$ are a sequence of polytopes, one in each dimension. Their faces correspond to sets of non-crossing diagonals in the $n$-sided polygon. In this case the polygon is fixed in the plane, with no rotations or flips allowed. However, we can still use the number of associahedron faces of each dimension to enumerate the vertices of the BME($n,k$) polytopes.

\begin{thm}\label{t:vertcount} The number of vertices of BME($n,k$), and thus the number of binary level-1 networks with $n$ leaves and $k$ non-trivial bridges, with $0 \le k \le  n-3$ is:
$$
\centering{v(n,k) = T(n, k)\frac{(n-1)!}{2^{k+1}}}$$ $$\text{ where } T(n,k) \text { gives the components of the face vector of the associahedron  } K(n). $$\end{thm}

 Here, as seen in entry A033282 of \cite{oeis}, $$ T(n,k) =  \left(\frac{1}{k+1}\right){n-3 \choose k}{n+k-1 \choose k}
$$
which allows the simpler count:
$$
v(n,k) = {n-3 \choose k}\frac{(n+k-1)!}{(2k+2)!!}
$$
 Table~\ref{counts} shows the number of vertices of BME($n,k$), (the number of binary level-1 networks with $n$ leaves and $k$ bridges) for small values  of $n,k.$ Note that the cases $k=0$ and $k=n-3$ count the circular orderings and phylogenetic trees of length $n$, respectively.
\begin{proof}
Vertices in the polytope BME($n,k$) correspond to the binary level-1 networks with $n$ leaves and $k$ nontrivial bridges. The bridges in a level-1 network $L=L(s)$ are the same as the bridges in any preimage split network $s.$ We construct a binary level-1 network $N$ as follows. We start with the $n$-sided polygon, label one side
as 1, and then label the remaining sides in any order.  The number of circular orderings is  $(n-1)!/2.$ Then, independently, we choose non-crossing diagonals for the $k$ bridges, counted by $T(n,k)$. Crucially, the side labeled 1 can be though of as the rooted edge so that each of the subdivisions of the polygon counted by $T(n,k)$ is actually a distinct choice. Since we are only counting up to twists around each bridge , we need to divide by $2^k.$ That completes the counting: next we construct the final binary level-1 network $N$ by making a graph cycle out of each region in the subdivided polygon (except the triangular regions, which become tree-like degree-3 nodes), and attaching bridges and leaves according to the labeled polygon edges and diagonals.  Equivalently, we can see the polygonal picture of $\Sigma(N)$ by adding crossing diagonals to each region of our subdivided polygon.
\end{proof}

Figure~\ref{f:constructo} shows the process of constructing a binary level-1 network using a cyclic ordering and a face of the associahedron. In \cite{Pachter2} the authors point out that a vertex of the associahedron, together with a cyclic ordering, corresponds to a phylogenetic tree. We see that the correct extension of that correspondence is to the binary level-1 phylogenetic networks.

\begin{figure}[h]
\centering{\includegraphics[width=\textwidth]{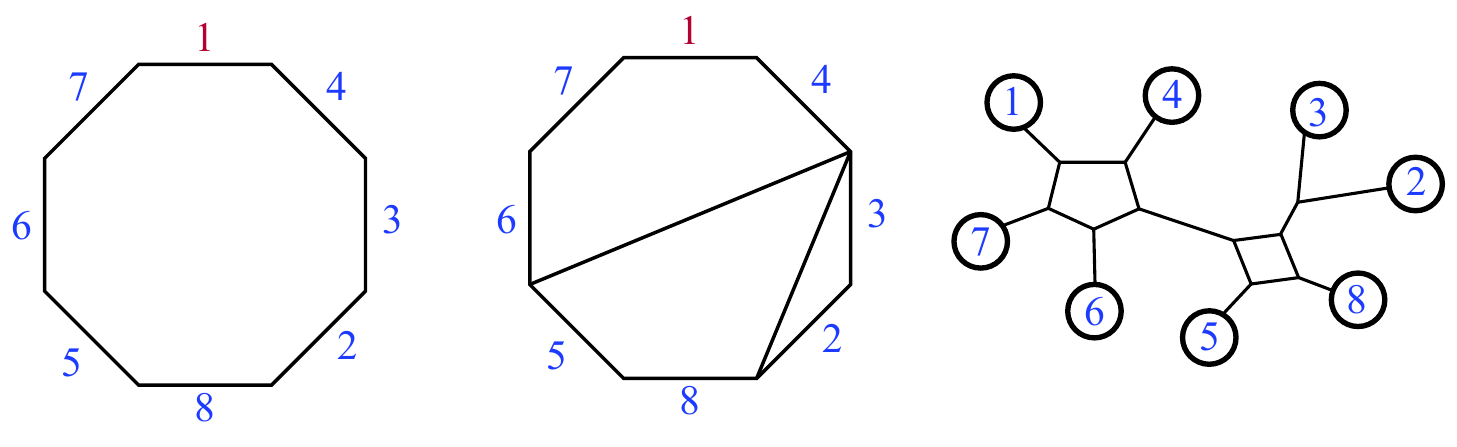}}
\caption {Starting with a cyclic ordering of the edges of the octagon, and adding bridges (to get a face of the associahedron) to construct a binary level-1 network.}
\label{f:constructo}
\end{figure}

 Also note that Semple and Steel use a generating function of three variables to derive a more general formula for counting galled trees in \cite{semple_steel_uni}. Thus their formula implies ours in the case that no cycles of length three are allowed. It is suggestive that their proof method uses Lagrange inversion of the generating function. As shown in \cite{ardila} Lagrange inversion of a series uses the face numbers of the associahedra. Here that fact instead allows us to enumerate the binary level-1 networks with $n$ taxa directly, by summing our formula from $k=0$ to $k=n-3$. The row sums of Table~\ref{counts} are 1, 6, 57, 750, 12645, 260190... which is sequence A032119 of \cite{oeis}. However it does not appear that that sequence has been used to count level-1 networks; rather it is described as counting rooted planar trees with labeled leaves, with equivalence under rotating subtrees left to right at any given branch node. The bijection from binary level-1 networks to these planar trees is straightforward, by replacing our cycles with nodes of higher degree.
 %(Mention Steel's other vector for high degree)

\begin{table}[hb!]
\begin{tabular}{|c|cccccccc|}
\hline
$n=$& $k=$& 0&1&2&3&4&5&6 \\
\hline
3&&1&&&&&&\\
4&&3 & 3&&&&&\\
5&&12 &30 &15&&&&\\
6&& 60& 270 &315 & 105&&&\\
7&& 360& 2520   & 5040  & 3780  &945&&\\
8&& 2520&   25200&  75600&  94500&  51975&  10395&\\
9&&20160&   272160& 1134000&    2079000&    1871100&    810810  &135135\\
  \hline
\end{tabular}\caption{Numbers of vertices for the BME($n,k$) polytopes, that is, numbers of binary level-1 phylogenetic networks with $n$ leaves and $k$ bridges. \label{counts}}
\end{table}

\section{Facets}\label{s:facet}

 We have found that many of the known facets of BME($n$) have analogues in BME($n,k$).

\begin{thm}\label{split_facet}
Any split $A|B$ of $[n]$ with $|A|>1$, $|B|>1$ corresponds to a face of BME($n,k$), for all $n,k$ with $ k\le n-3$.  The vertices of that face are the binary level-1
 networks which display the split $A|B.$ Furthermore, if $|A|>2$, $|B|>2$, the face is shown to be a facet of the polytope.
 The face inequality is:
 $$
 \sum_{i,j\in A} x_{ij} \le (|A|-1)2^k.
 $$\end{thm}

\begin{proof}
First we show that the collection of vertices corresponding to networks displaying a split $A|B$ of $[n]$ obey our linear equality, and that all other vertices obey a corresponding inequality.

 The equality follows from similar logic to that in the proofs of Theorem~\ref{cyc_vec} and Corollary~\ref{compsum}. If a network displays a split $A|B$ then so must every circular ordering of $[n]$ consistent with that network. There are $2^k$ such circular orderings. If a circular ordering $c$ displays the split $A|B$ then it must have the leaves of $A$  contiguous in that circular ordering. Thus the components $x_{ij}(c)$ with $i,j$ both corresponding to leaves from $A$ will contain exactly $|A|-1$ entries that are equal to 1, for the edges connecting those leaves. The equality follows: for $s$ displaying the split $A|B$ we have $\sum_{i,j\in A} x(s)_{ij} = (|A|-1)2^k.$

The strict inequality holds for any binary level-1 network which does not display the split $A|B.$ The same reasoning as above holds, but this time there will be at least one circular ordering $c$ which does not display the split. In that circular ordering, the components corresponding to the leaves from $|A|$ have fewer 1's than the maximum $|A|-1$, so the sum will be less.

Next we show that the face $F_A(n,k)$ just described is of codimension 1 when both parts of the split are larger than 2. The polytopes are all of dimension  ${n \choose 2} - n$, so we show the facets are of dimension 1 less.
We use the fact that for any polytope, its scaling by $m$ is of the same dimension. Equivalently, taking sets of $m$ vectors all from the same facet of a polytope, the vector sums of $m$ such vectors will all lie in an affine space of the same dimension as that facet. (Thus any subset of those sums of $m$ vectors each will have a convex hull of smaller or equal dimension than the original facet.)

We have that a given split, with both parts larger than 2, corresponds to a subtour-elimination facet of STSP($n$) = BME($n,0$) and also to a split-facet of BME($n,n-3$) = BME($n$). Both are dimension ${n\choose 2}-n-1.$ From Lemma~\ref{twist}, we know that each vertex $\mathbf{x}(s)$ of the proposed split-facet $F_A(n,k)$ is the vector sum $\mathbf{x}(s')+ \mathbf{x}(s'')$. The  two summands are both vertices of the proposed split-facet $F_A(n,k-1),$ since $s'$ and $s''$  both display the split $A|B.$

%Specifically for a vertex $\mathbf{x}(s)$ of $F_A(n,k)$ we choose one bridge $e$ of $s$. After adding a split to the network which crosses  that bridge and no other bridges, the result still contains the split $A|B.$ We call this fixing the bridge $e$. We do so in two separate ways, one for each orientation around that bridge $e$, to get $s'$ and $s''$ both refinements of $s.$ In each, the components for leaves not adjacent to $e$ will be just  $2^{k-1-b_{ij}}$ if the fixed bridge $e$ is not on the path between them (or 0 if  they cannot be adjacent.) The the sum for such a component is $2*2^{k-1-b_{ij}} = 2^{k-b_{ij}}$ as claimed. For a component with $e$ on the path between the leaves $i,j$ which can be adjacent we have that one fixing of the bridge $e$ no longer allow adjacency, while the other fixing has component $2^{k-1-(b_{ij}-1)}.$ Thus again the sum is as claimed.

Therefore the facet $F_A(n,k)$ cannot be of greater dimension than the  facet $F_A(n,k-1).$ Since the dimension cannot increase at any step between $k=0$ and $k=n-3$, and it has the same value for $k=0$ and for $k=n-3$,  then it must remain constant for each $k$ at ${n \choose 2} - n - 1.$
\end{proof}

For the case of splits with one part of size two, we know that these do correspond to a facet when $k=0$, the STSP($n$), but not when $k=n-3$, in BME($n$). It is an open question for which other $n,k$ the splits of size two correspond to facets of BME($n,k$.) We conjecture this for all $k<n-3,$ but we can only report the positive result for $n=5.$

Figure~\ref{f:filter}(a) shows a subtour elimination facet  of STSP(5) = BME(5,0), corresponding to the split $s=\{\{1,2\},\{3,4,5\}\}.$ In this case it is combinatorially equivalent to the $4D$ Birkhoff polytope.  Split networks label subfaces of this facet. Figure~\ref{f:filter}(b) shows the corresponding split-facet of BME(5,1), in which the vertices are nine networks with a single bridge that each refine $s$. Figure~\ref{split_facet} shows the same split-facet, with alternate labels. Figure~\ref{f:filter}(c) shows the 2D face of BME(5) corresponding to the same split. Summing either of the horizontal or  vertical pairs of vectors shown in (a) gives the four vectors shown in (b).  Summing all four vectors in (a) gives the vector shown in (c).

Next we look at the existence of lower bound faces, and conjecture that they are in fact facets as well.

\begin{thm}\label{lb}
For each pair $i,j \in [n]$ we get a face of BME($n,k$) for all $n,k$ with $ k\le n-3$. For $k=n-3$ these are the caterpillar facets. For $k\le n-4$ these contain the networks with no consistent circular orderings such that $i,j$ are adjacent. The face inequality for each of these latter is $x_{ij}\ge 0.$
 \end{thm}
\begin{proof}
For $k\le n-4$ the equality $x_{ij} =  0$ clearly holds by definition for networks with no consistent circular orderings such that $i,j$ are adjacent. For any network which possesses a consistent circular ordering with $i,j$ adjacent, we see the component $x_{ij} > 0.$
\end{proof}

We would like to know which of these lower bound faces are facets. They are facets for the case of $k=0,$ the Symmetric Travelling Salesman polytopes, and for $k=n-3,$ the Balanced Minimum Evolution polytopes. For the case of $k=n-3$ however the caterpillar facets have vertices that obey $x_{ij}\ge 1.$ In fact none of their vector components are zero, so they are not the sum of a pair of lower bound face vertices from a given lower bound face. However the lower bound face is a facet for $n=5, k=1$, and we conjecture this is true for all lower bound faces.

\section{BME(5,1)}\label{5:1}
We have investigated more fully the case of BME(5,1), by using polymake to find all the facets and then observing the patterns they obey. There are 62 facets altogether, of four different types. The $(5,1)$ networks are especially simple: each has the same underlying graph, with a cherry clade attached to a length-4 cycle with one central leaf across from the cherry.

\begin{thm}\label{2splits}
In BME(5,1) there are ${5 \choose 2} = 10$ split facets, each of which has nine vertices.
\end{thm}
\begin{proof}These are predicted to be faces by Theorem~\ref{cool}, but it is a surprise that they are indeed facets since each nontrivial split of $[5]$ has a part of size 2.  We check that they are facets by inspection in polymake.    Let the split be $\{a,b\}|\{c,d,e\}$. Each split facet in BME($5,1$) has nine vertices: These are the networks formed by 3 ways to put $\{a,b\}$ on the cherry with a choice of the other 3 on the central leaf, 3 ways to put $a$ on the central leaf, and 3 ways to put $b$ on the central leaf. \end{proof}
From polymake, we find that these lower bound facets are each a product of two triangles. For an example see the split facet for the split $\{1,2\}|\{3,4,5\}$ pictured in Figure~\ref{split_facet}, where the inequality is $x_{1,2} \le 2$. Compare to Figure~\ref{f:filter} where the same facet is shown labeled by polygons.

\begin{figure}[h]
\centering{\includegraphics[width=\textwidth]{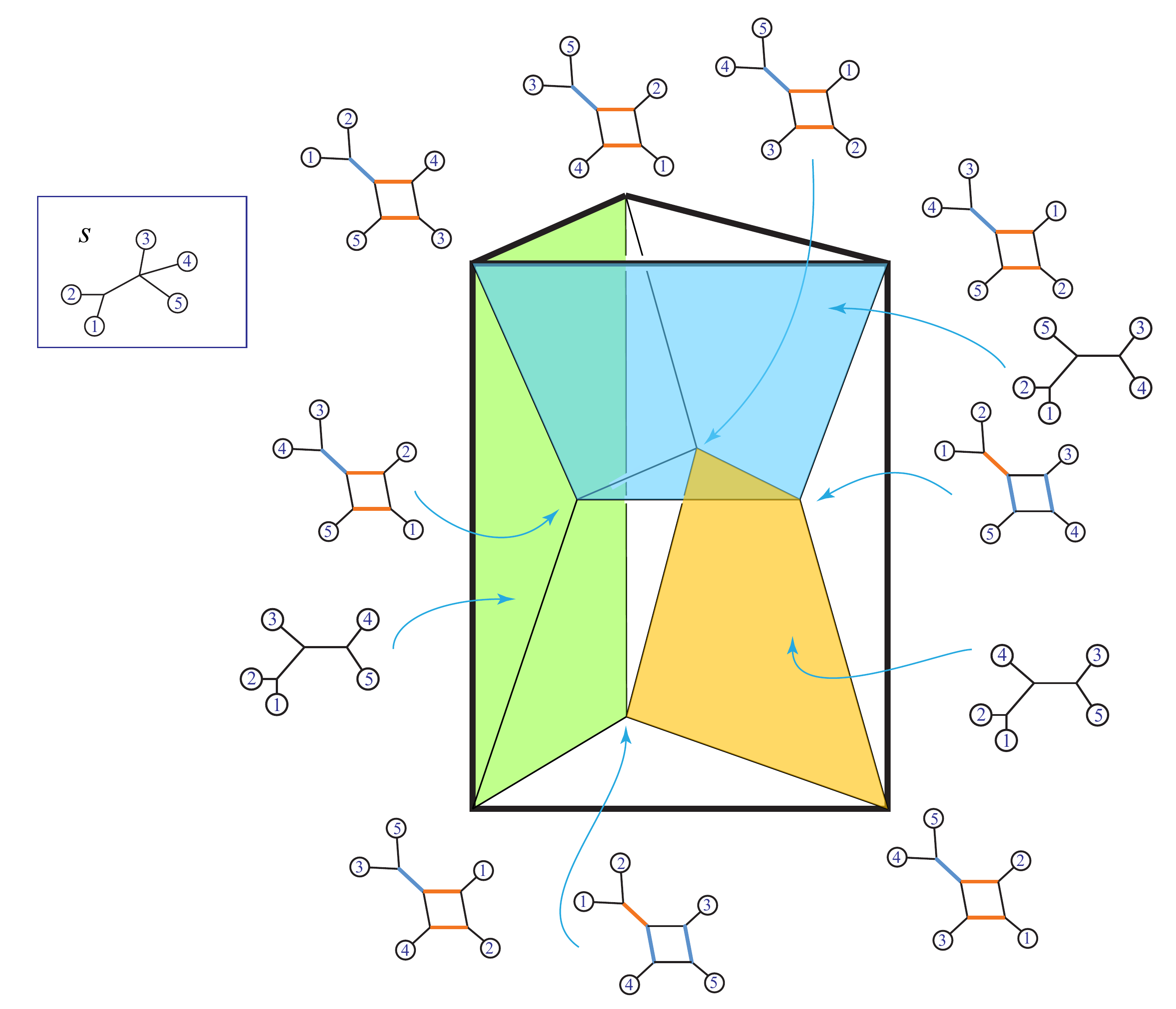}}
\caption {A split facet in BME(5,1). These vertices obey $x_{1,2} = 2$}
\label{split_facet}
\end{figure}

\begin{thm}\label{lb_51} In BME(5,1) there are 10 lower bound facets, one for each component of $\mathbf{x}(s)_{ij}.$ Each has nine vertices.
\end{thm}
\begin{proof} These are predicted to be faces by Theorem~\ref{lb}, but here again we see they are indeed facets via polymake. Each lower bound facet in BME($5,1$) has nine vertices: three ways to put $i$ on the central leaf and $j$ on the cherry with one of the other three taxa, three ways to put $j$ on the central leaf, and three ways to put $i,j$ non-adjacent on the 4-cycle with one of the other three on the central leaf.
\end{proof}

From polymake, we find that these lower bound facets are each a product of two triangles. For an example see the lower bound facet for the inequality $x_{1,2} \ge 0$ pictured in Figure~\ref{lb_facet}.

\begin{figure}[h]
\centering{\includegraphics[width=\textwidth]{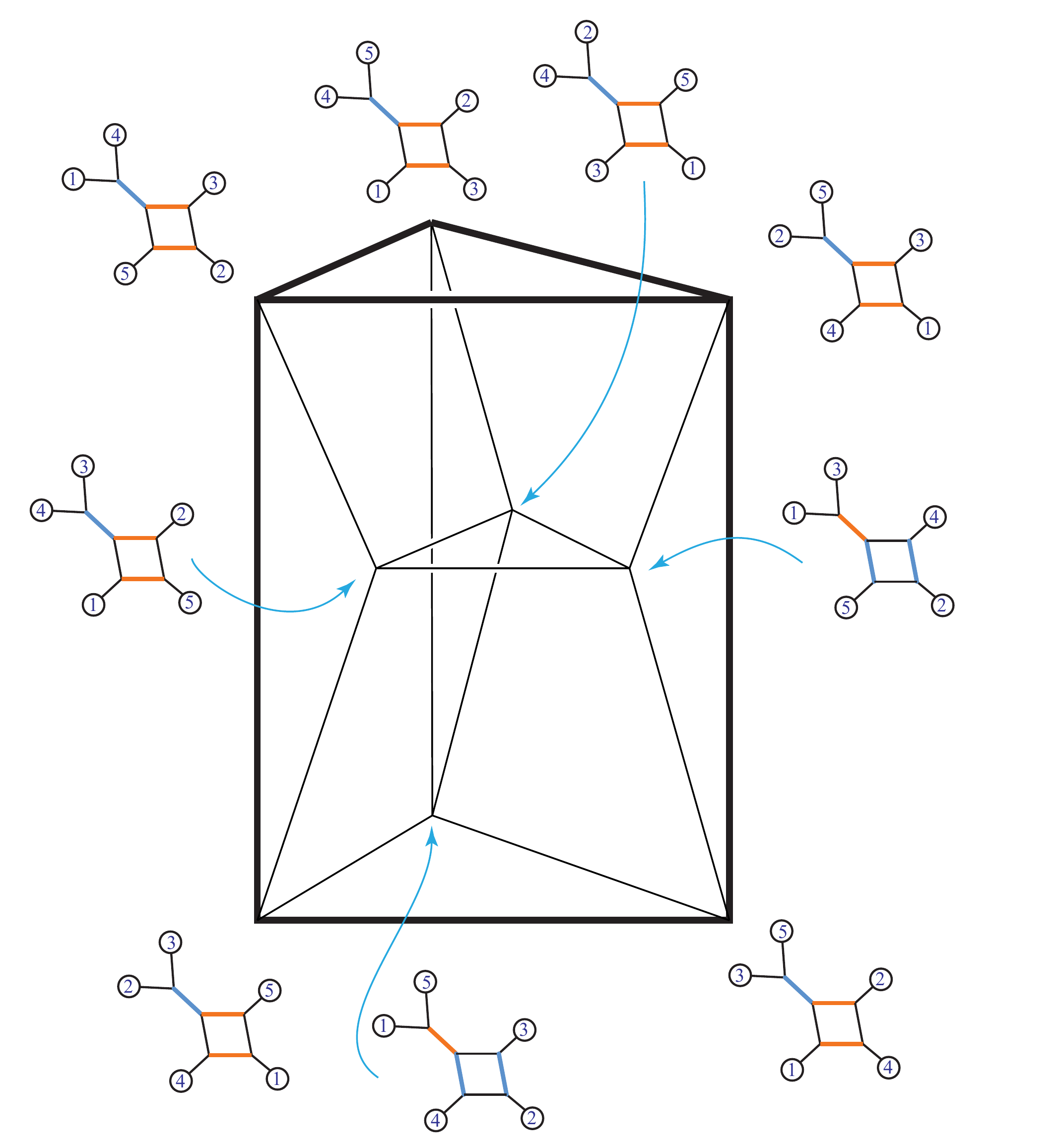}}
\caption {A lower bound facet in BME(5,1). These vertices obey $x_{1,2} = 0.$ }
\label{lb_facet}
\end{figure}

\begin{thm}\label{exclud} In BME(5,1) there are 30 facets which we call the \emph{excluded node} facets. They have 8 vertices each, which obey (sharply) the facet inequality: $$x_{ab} + x_{cd}-x_{ac}-x_{bd} \le 3$$ where $a,b,c,d,$ are four taxa in cyclic order.
\end{thm}
\begin{proof} These facets correspond to choosing 4 of the 5 taxa, excluding one taxon. Then the four are given a cyclic order, and then that cyclic order is split into two contiguous pairs. These choices are independent, giving 5(3)2 = 30 facets. Each facet has 8 networks as its vertices, found by choosing to place the excluded taxon on either the cherry or the 4-cycle (but not on the central leaf); followed by placing the chosen four taxa on the remaining leaves in their cyclic order, but not allowing either contiguous pair to be separated by the excluded taxon.
By inspection, each excluded node facet obeys the following inequality, where the circular ordering is $a,b;c,d$ with the first and second pairs contiguous:
$x_{ab} + x_{cd}-x_{ac}-x_{bd} \le 3$.\end{proof}
Each excluded node facet is a 4D prism: the interval crossed with a tetrahedron. For an example see the excluded node facet for the cycle $1,2;5,3$ pictured in Figure~\ref{excyc}.

\begin{figure}[h]
\centering{\includegraphics[width=\textwidth]{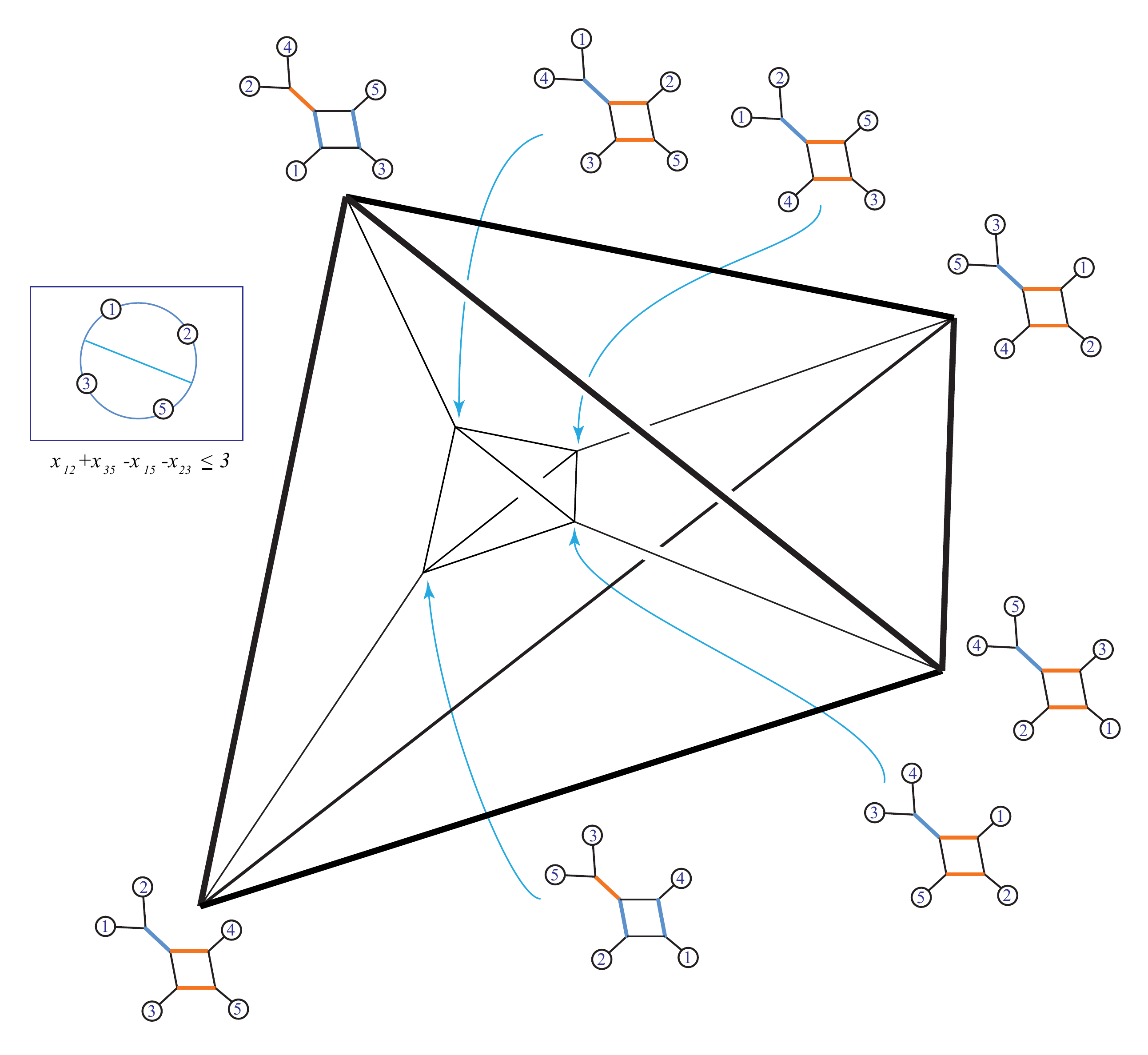}} \caption
{An excluded node facet in BME(5,1). Vertices obey  $x_{1,2} +
x_{5,3}-x_{1,5}-x_{2,3} = 3.$ } \label{excyc}
\end{figure}

\begin{thm}\label{cyc51}
In BME(5,1) there are 12 facets which we call the \emph{cyclic order} facets. They have 5 vertices each, which obey (sharply) the facet inequality: $$x_{ab} + x_{bc} + x_{cd} + x_{df} + x_{af} \le 8.$$ where $a,b,c,d,f$ is a cyclic order on the taxa. \end{thm}
\begin{proof}
By inspection, there is a cyclic order facet for each of the 12 circular orderings. Each network (vertex) in
the facet represents the same circular ordering when reading the leaves around the network in a
circle (up to twists around the bridge). There are five such networks for a given circular ordering, distinct by the choice of which taxa to place on the central leaf of the 4-cycle. By inspection, the vectors of the trees in these facets adhere to the equality
$x_{ab} + x_{bc} + x_{cd} + x_{df} + x_{af} = 8.$
\end{proof}
For an example see the cyclic order facet for the cycle $a,b,c,d,f$ pictured in Figure~\ref{cyc}.

\begin{figure}[h]
\centering{\includegraphics[width=\textwidth]{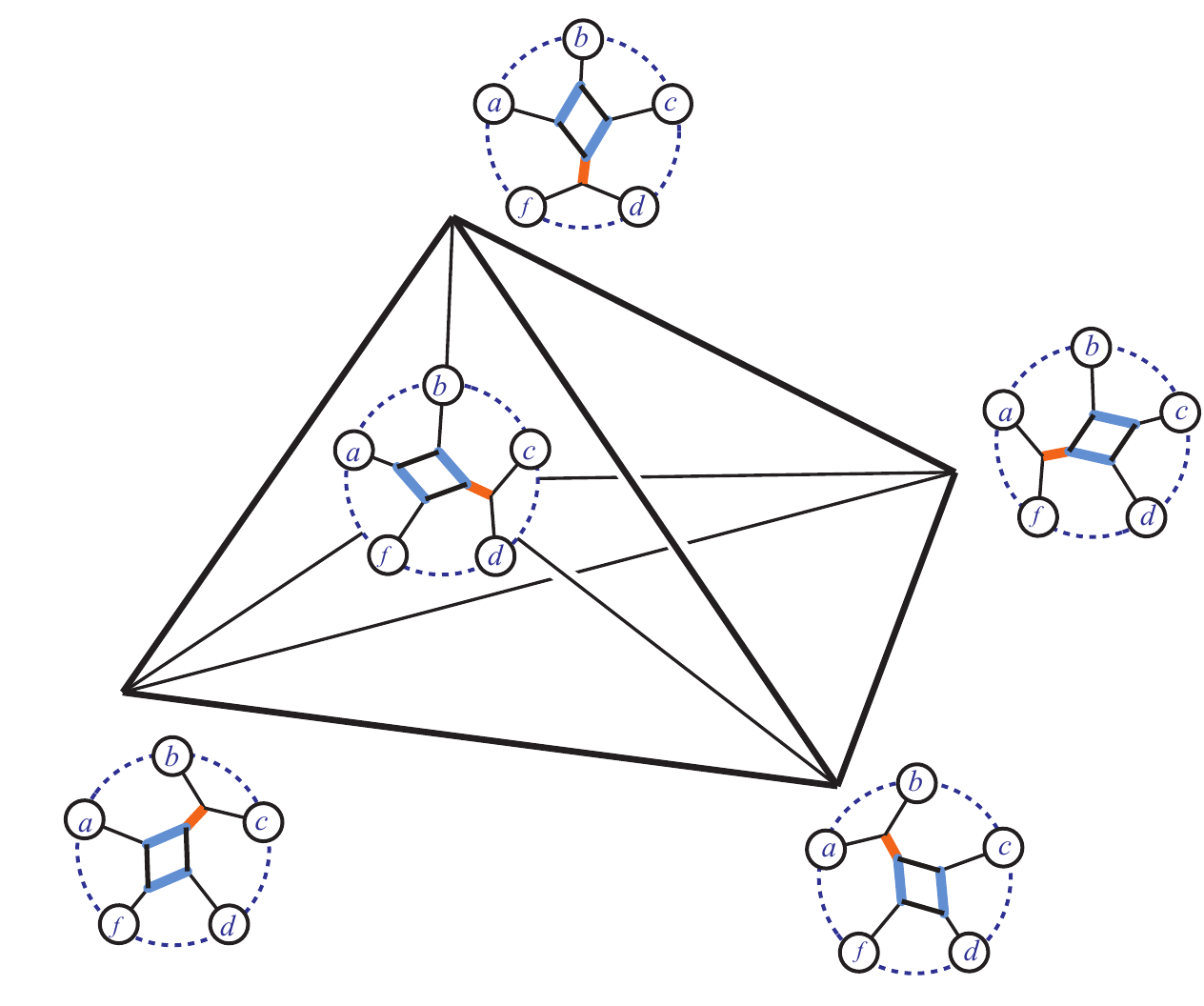}}
\caption {A generic cyclic order facet in BME(5,1). These vertices obey $x_{ab} + x_{bc} + x_{cd} + x_{df} + x_{af} = 8.$}
\label{cyc}
\end{figure}

%%%%%%%%%%%%%%%%%%%%%%%%%%%%%%%%%%%%%%%%%%%%%%%%%%%%%%%%%%%%%%%%%%%%%%%%%%%%%%%%%%%%%

%%%%%%%%%%%%%%%%%%%%%%%%%%%%%%%%%%%%%%%%%%%%%%%%%%%%%%%%%%%%%%%%%%%%%%%%%%%%%%%%%%%%%
%
%                  References
%
%%%%%%%%%%%%%%%%%%%%%%%%%%%%%%%%%%%%%%%%%%%%%%%%%%%%%%%%%%%%%%%%%%%%%%%%%%%%%%%%%%%%%

\bibliographystyle{amsplain}
\bibliography{phylogenetics}{}

\providecommand{\bysame}{\leavevmode\hbox to3em{\hrulefill}\thinspace}
\providecommand{\MR}{\relax\ifhmode\unskip\space\fi MR }
% \MRhref is called by the amsart/book/proc definition of \MR.
\providecommand{\MRhref}[2]{%
  \href{http://www.ams.org/mathscinet-getitem?mr=#1}{#2}
}
\providecommand{\href}[2]{#2}
\begin{thebibliography}{10}

\bibitem{ardila}
M.~Aguiar and F.~Ardila, \emph{Hopf monoids and generalized permutahedra},
  arXiv:1709.07504 [math.CO] (2017), 1--113.

\bibitem{brandes}
Ulrik Brandes and Sabine Cornelsen, \emph{Phylogenetic graph models beyond
  trees}, Discrete Appl. Math. \textbf{157} (2009), no.~10, 2361--2369.
  \MR{2527953}

\bibitem{Bryant2007}
David Bryant, Vincent Moulton, and Andreas Spillner, \emph{Consistency of the
  neighbor-net algorithm}, Algorithms for Molecular Biology \textbf{2} (2007),
  no.~1, 8.

\bibitem{dev-petti}
S.~Devadoss and S.~Petti, \emph{A space of phylogenetic networks}, SIAM Journal
  on Applied Algebra and Geometry \textbf{1} (2017), 683--705.

\bibitem{Rudy2008}
K.~Eickmeyer, P.~Huggins, L.~Pachter, and R.~Yoshida, \emph{On the optimality
  of the neighbor-joining algorithm}, Algorithms for Molecular Biology
  \textbf{3} (2008), no.~5.

\bibitem{forcey2015facets}
S.~Forcey, L.~Keefe, and W.~Sands, \emph{Facets of the balanced minimal
  evolution polytope}, Journal of Mathematical Biology \textbf{73} (2016),
  no.~2, 447--468.

\bibitem{splito}
S.~Forcey, L.~Keefe, and W.~Sands, \emph{Split-facets for balanced minimal
  evolution polytopes and the permutoassociahedron}, Bulletin of Mathematical
  Biology \textbf{79} (2017), no.~5, 975--994.

\bibitem{Gambette2017}
P.~Gambette, K.~T. Huber, and G.~E. Scholz, \emph{Uprooted phylogenetic
  networks}, Bulletin of Mathematical Biology \textbf{79} (2017), no.~9,
  2022--2048.

\bibitem{Rudy}
D.~Haws, T.~Hodge, and R.~Yoshida, \emph{Optimality of the neighbor joining
  algorithm and faces of the balanced minimum evolution polytope}, Bull. Math.
  Biol. \textbf{73} (2011), no.~11, 2627--2648. \MR{2855185 (2012h:92108)}

\bibitem{Pachter2}
D.~Levy and Lior Pachter, \emph{The neighbor-net algorithm}, Advances in
  Applied Mathematics \textbf{47} (2011), 240--258.

\bibitem{semple_steel_uni}
C.~{Semple} and M.~{Steel}, \emph{Unicyclic networks: compatibility and
  enumeration}, IEEE/ACM Transactions on Computational Biology and
  Bioinformatics \textbf{3} (2006), no.~1, 84--91.

\bibitem{Steel}
Charles Semple and Mike Steel, \emph{Cyclic permutations and evolutionary
  trees}, Adv. in Appl. Math. \textbf{32} (2004), no.~4, 669--680. \MR{2053839
  (2005g:05042)}

\bibitem{oeis}
N.~J.~A. Sloane, \emph{The on-line encyclopedia of integer sequences}, 2018,
  published electronically at {\tt www.oeis.org}.

\bibitem{steelphyl}
Mike Steel, \emph{Phylogeny---discrete and random processes in evolution},
  CBMS-NSF Regional Conference Series in Applied Mathematics, vol.~89, Society
  for Industrial and Applied Mathematics (SIAM), Philadelphia, PA, 2016.
  \MR{3601108}

\end{thebibliography}

\end{document}